\documentclass[10pt,a4paper,twoside,reqno]{amsart} 
\usepackage{amsfonts,amssymb,amscd,amsmath,enumerate,verbatim,calc} 
\usepackage[utf8]{inputenc}
\usepackage[T1]{fontenc}
\usepackage{amsthm}
\usepackage{float}

\newtheorem{theorem}{Theorem}

\newtheorem{cor}{Corollary}
\newtheorem{question}{Question}

\theoremstyle{definition}
\newtheorem*{example}{Example}

\title{Note on vanishing power sums of roots of unity}
 \author{Neeraj Kumar and K. Senthil Kumar}
\address{The Institute of Mathematical Sciences\\
4th cross road, CIT Campus, Taramani\\
Chennai 600113, India}
\email{neerajkr@imsc.res.in,  senthilkk@imsc.res.in}

\thanks{{\it Date:} Submitted: March $30,\;2015$, Revision submitted: May $04, \; 2015$.}
\thanks{{\it Key words:} roots of unity, power sum.
\endgraf
{\it 2000 Mathematics Subject Classification:} Primary $11L03$ Secondary $11R18$}

\begin{document}

 \begin{abstract}
 For fixed positive integers $m$ and $\ell$, we  give a complete list of integers $n$ for which their exist $m$th complex roots 
 of unity $x_1,\dots,x_n $ such that  $x_1^{\ell} + \cdots + x_n^{\ell}=0$. This extends the earlier result 
 of Lam and Leung on vanishing sums of roots of unity. Furthermore, we characterize all positive integers $n$ with $2 \leq n \leq m$, 
 for which there are distinct $m$th complex roots of unity $x_1,\dots,x_n $ such that  $x_1^{\ell} + \cdots + x_n^{\ell}=0$.
 \end{abstract}

 \maketitle

\section{Introduction}

Let $m$ be a positive integer. By an $m$th root of unity, we mean a complex number $\zeta$ such that $\zeta^m=1.$ That is, a 
root of the polynomial $X^m-1.$ One can easily see that the roots of $X^m-1$ are distinct, in fact there are exactly $m$, 
$m$th roots of unity. Using the relationship between the roots and the coefficients of a polynomial, we see that the sum 
of all $m$th roots of unity, which is the coefficient of $X^{m-1}$ in $X^m-1,$ is zero.  A natural question is: What are all 
the positive integers $n$ for which there exist $m$th roots of unity $x_1,\ldots,x_n$ (repetition is allowed) such that 
$x_1+\cdots+x_n=0$. A beautiful result of T. Y. Lam and K. H. Leung \cite{lam} gives a complete classification of all such integers. 
Suppose $m$ has prime factorization $p_1^{a_1}\ldots p_r^{a_r}$, where $a_i >0$, then we have the following theorem due to 
Lam and Leung:
\begin{theorem}\label{thm:1}
 Let $n$ be a positive integer. Then there are $m$th roots of unity $x_1,\ldots,x_n$ such that 
 $x_1+\cdots+x_n=0$ if and only if $n$ is of the form $n_1p_1+\cdots+n_rp_r$ where each $n_i$ is a non-negative integer 
 for $1\leq i \leq r.$
\end{theorem}

\noindent Theorem \ref{thm:1} motivate us to ask the following:

\begin{question}\label{lam-ext-quest}
 Let $m$ and $\ell$ be positive integers. What are all the positive integers $n$ for 
which there exist $m$th roots of unity $x_1,\ldots,x_n$ such that $x_1^{\ell} + \cdots + x_n^{\ell}=0$ ?
\end{question}


Note that when $\ell =1$, the complete answer to Question \ref{lam-ext-quest} is given by Theorem \ref{thm:1}. 
However, for $ \ell  \geq 2$, we do not find any results in this direction in the literature. Our objective 
here is to study the case when $ \ell  \geq 2$. First, we fix some notations. Let $m$ be a positive integer, and 
let $\Omega_m$ denotes the set of all $m$th roots of unity. 
For a positive integer $\ell,$ $ W_\ell(m)$ denotes the set of all positive integers $n$ for which 
there exist $n$-elements $x_1,\dots,x_n \; \in \;\Omega_m $ such that $x_1^{\ell} + \cdots + x_n^{\ell}=0$. 
When ${\ell}=1$, we simply denote $ W_\ell(m)$ by $W(m)$. With this notation, Question \ref{lam-ext-quest} can be 
reformulated as follows: { \it  Let $m$ and $\ell$ be positive integers. What are all the positive integers in the set $W_\ell(m)$ ?}

It is clear that if $m$ divides ${\ell}$ then $W_\ell(m)$ is an empty set. Suppose that there are $m$th complex 
roots of unity, say, $x_1,\dots,x_n$ such that $x_1^\ell+\cdots+x_n^\ell=0.$  Since the $\ell$th power of an $m$th root of unity is
still an $m$th root of unity, the equation $x_1^\ell+\cdots+x_n^\ell=0$ with $x_i \in \Omega_m$ can be written in the 
form $y_1+\cdots+y_n=0$ with $y_i \in \Omega_m$. This shows that for any positive integer $m$ and $\ell$, $W_\ell(m)$ is a subset of $W(m)$. It 
follows from Theorem \ref{thm:1} that any positive integers in the set $W_\ell(m)$ must be of the form $n_1p_1+\cdots+n_rp_r$ where 
each $n_i$ is a non-negative integer for $1\leq i \leq r.$ In Section $2$, we give a complete list of integers in the set $W_\ell(m)$
( see Theorem \ref{thm:2}). Moreover, in Section $3$ we find all positive integers $n \in W_\ell(m)$ for which there are distinct $m$th complex roots of 
unity $x_1,\dots,x_n$ such that $x_1^\ell+\cdots+x_n^\ell=0$ ( see Theorem \ref{thm:3}).

There are algebraic aspects why Question \ref{lam-ext-quest} is important. 
For instance, for a positive integer $a$, denote by $p_a$ 
the power sum polynomial $X_1^a+\cdots+X_n^a$ of degree $a$. Let $ \ell < k$ be two positive integers. In commutative algebra, 
one encounters the following situation: To show that the ideal $\langle p_{\ell}, p_k \rangle$ generated by the polynomials 
$p_{\ell}$ and $p_k$ is a prime ideal in $\mathbb{C}[X_1,\dots,X_n]$, one needs to show that the power sum 
polynomial $X_1^{\ell} + \cdots + X_n^{\ell}$  does not vanish when one allows the $X_i$'s to take values among 
the $(\ell-k)$th roots of unity \cite[see proof of Theorem 3.8]{k}.

\vspace{2mm}

\paragraph{{\bf Acknowledgment:}} We are grateful to Professor Ram Murty for his valuable suggestions regarding the paper. 
We also thank the referee for many useful comments for improving the manuscript.  
This project was funded by the Department of Atomic Energy (DAE), Government of India.

\section{vanishing of power sums of roots of unity}\label{main-thm}

Let $m$ and ${\ell}$ be positive integers. In this section, we completely characterize all the positive integers in the set $W_\ell(m)$. 
More precisely, we prove the following theorem:

\begin{theorem}\label{thm:2}
Let $m$ and ${\ell}$ be positive integers. Let $d=(m,{\ell})$ be the greatest common divisor of $m$ and ${\ell}.$ 
Then $W_\ell(m)=W(m/d).$
\end{theorem}
In other words, Theorem \ref{thm:2} says that: For any positive integer $n,$  $x_1^{\ell} + \cdots + x_n^{\ell}=0$ 
with $x_i\in\Omega_m$ if and only if $y_1+\cdots+y_n=0$ with $y_i\in\Omega_{m/d}.$

\begin{proof} It is well known that $\Omega_m$, that is, the set of all $m$th roots of unity, form a group with respect 
to the multiplication of complex numbers. 
In fact, it is a cyclic group of order $m$, generated by the complex number $\zeta_m=\cos 2\pi/m+\mathfrak{i}\sin 2\pi/m.$ There is a 
remarkable property  about finite cyclic groups. Namely, if $G$ is a finite cyclic group and $l$ is a positive integer 
relatively prime to the order of $G,$ then the map 
\begin{eqnarray}\label{equ:1}
x  \longmapsto x^{\ell}    \quad ( x \; \in  \; G)
\end{eqnarray}
is an automorphism of $G$ (In fact, all the automorphisms of $G$ are of the form (\ref{equ:1}) for some integer 
${\ell}$ which is relatively prime to the order of $G$). It follows that, if ${\ell}$ is a positive integer which is relatively 
prime to $m$ then every element of $\Omega_m$ is a ${\ell}$th power of some element of 
$\Omega_m.$ Thus, for an integer $l$ which is relatively prime to $m,$ the equation $x_1^{\ell}+\cdots+x_n^{\ell}=0$ with $x_i\in \Omega_m$ can be replaced by $y_1+\ldots+y_n=0$ 
with $y_i\in \Omega_m,$ and vice versa. This discussion proves Theorem \ref{thm:2} for the case when $\ell$ is relatively prime to $m.$

Now assume that $d>1$. Consider the map 
\begin{eqnarray}\label{psi-map}
\psi_d: \Omega_m  \longrightarrow \Omega_{m/d}
\end{eqnarray}
defined by $x\mapsto x^d$ for $x\in\Omega_m.$ This map is clearly onto, and the kernel is 
exactly $\Omega_d.$ Thus, $\Omega_m/\Omega_d\cong \Omega_{m/d}.$ 
Now suppose that there are elements $x_1,\ldots,x_n\in\Omega_m$ such that $x_1^{\ell}+\cdots+x_n^{\ell}=0.$ Then this sum can be rewritten 
as $\left(x_1^{{\ell}/d}\right)^d+\cdots+\left(x_n^{{\ell}/d}\right)^d=0.$ Since ${\ell}/d$ and $m$ are relatively prime, 
by the above discussion, the latter equation can be rewritten in the form $y_1^d+\cdots+y_n^d=0$ with $y_i\in\Omega_m.$ 
Finally, using the map $\psi_d,$ the latter sum can be realized 
as $z_1+\ldots+z_n=0$ where $z_i\in\Omega_{m/d}$ for $1\leq i\leq n.$ In fact, all these steps can be reversed. 
This completes the proof of Theorem \ref{thm:2}.
\end{proof}

Combining Theorems \ref{thm:1} and \ref{thm:2}, we have the following corollary:

\begin{cor}\label{cor:1} 
Let $m$, $n$ and ${\ell}$ be positive integers. Let $d=(m,{\ell})$ be the greatest common divisor of $m$ and ${\ell}.$ Then there 
are $m$th roots of unity $x_1,\ldots,x_n$ such that $x_1^\ell+\cdots+x_n^\ell=0$ if and only if $n$ is of the form $n_1q_1+\cdots+n_sq_s$ where 
each $n_i$ is a non-negative integer for $1\leq i \leq s$ and $q_1,\dots,q_s$ are distinct prime divisors of $m/d$.
 
\end{cor}

\begin{example} Let $m=60$, and let $\ell$ be an integer with $1 \leq  \ell < 60$. 
By Theorem \ref{thm:2}, $W_\ell(m)=W(m/d)$ where $d$ is the greatest common divisor of $m$ and $\ell$. 
When $d$ varies over the divisors of $m$, $m/d$ also varies over the divisors of $m$. Thus $W_\ell(m)$ coincides with $W(d)$ for some divisor $d$ of $m$. 
On the other hand, by Theorem \ref{thm:1}, $W(d)=\sum_{i=1}^{s}q_i\mathbb{N}$ 
where $d=q_1^{b_1}\dots q_{s}^{b_s}$ is the prime factorization of $d$. Here $\mathbb{N}$ denotes the set of non-negative integers. We thus have the following table 
 which describe $W(d)$ for all positive divisors $d$ of $m=60$.
\end{example}

{ \footnotesize

\begin{table}[H]
\centering
\begin{tabular}{| c | l |}
 \hline 
 $d$ & $W(d)$ \\ \hline
  $1$ & $\emptyset$ \\ 
  $2$ & $2 \mathbb{N}$ \\
  $3$ & $3 \mathbb{N}$ \\
  $4$ & $2 \mathbb{N}$ \\
  $5$ & $ 5 \mathbb{N}$ \\
  $6$ & $ 2 \mathbb{N}$ + $3 \mathbb{N}  = \mathbb{N} \setminus \{ 1 \}$ \\
  $10$ & $ 2 \mathbb{N}$ + $5 \mathbb{N}  = \mathbb{N}  \setminus  \{1,3 \}$ \\
  $12$ & $ 2 \mathbb{N}$ + $3 \mathbb{N}  = \mathbb{N} \setminus \{1 \}$ \\
  $15$ & $ 3 \mathbb{N}$ + $5 \mathbb{N}  = \mathbb{N} \setminus\{ 1,2,4,7 \}$ \\
  $20$ & $ 2 \mathbb{N}$ + $5 \mathbb{N}  = \mathbb{N}  \setminus \{ 1,3 \}$ \\ 
  $30$ & $ 2 \mathbb{N}$ + $3 \mathbb{N} $ + $5 \mathbb{N}  = \mathbb{N}  \setminus \{ 1 \}$ \\ 
  $60$ & $ 2 \mathbb{N}$ + $3 \mathbb{N} $ + $5 \mathbb{N}  = \mathbb{N} \setminus \{ 1 \}$ \\  
 \hline
\end{tabular}
\label{table}
\end{table}
}

\section{vanishing of power sums of distinct roots of unity}


Let $m$ and $\ell$ be two positive integers. For an integer $n \in W_{\ell}(m)$, the {\it height} $H(n;\ell,m)$ of $n$ is defined to be 
the smallest positive integer $h$ for which there are $m$th roots of unity $x_1,\dots,x_n$ such that $x_1^\ell+\cdots+x_n^\ell=0$ 
and the maximum among the repetition of $x_i$'s is $h$, 
that is, $h$ is the maximum among the $h_i$, where $h_i$ is the number of times $x_i$ appears in the list $x_1,\dots,x_n$.
When $\ell =1$, we denote $H(n;\ell,m)$ by $H(n;m)$. Note that $H(n;m)=1$ provided $ 2 \leq n \leq m$. Gary Sivek \cite{Gary} refined the 
work of Lam and Leung by proving that for any integers $m \geq 2$ and $2 \leq n \leq m$, $H(n;m)=1$ if 
and only if both $n$ and $m-n$ are expressible as a linear combination of the prime factors of $m$ with non-negative 
integer coefficients. Here we extend Sivek's result to vanishing of power sums of distinct roots of unity:

\begin{theorem}\label{thm:3}
Let $m$ and $l$ be positive integers, and let $n$ be an integer such that $2 \leq n\leq m.$ Let $d$ be the greatest common divisor of $m$ and $\ell$. 
Then $H(n;\ell,m)=1$ if and only if $H(n; m/d) \leq d$. 
\end{theorem}
\begin{proof}
Let $\Omega_{m/d} = \{ z_1,\dots,z_{m/d} \}$. Suppose that there are distinct $m$th roots of unity $x_1,\dots,x_n$ such 
that $x_1^\ell+\cdots+x_n^\ell=0$. Since $d$ is the greatest common divisor 
of $\ell$ and $m$, this equation can be rewritten in the form $y_1^d+\cdots+y_n^d=0$ with $y_1,\dots,y_n$ are $m$th roots of unity. 
Using the map $\psi_d$, the latter equation can be written as $\sum_{i=1}^{m/d} a_iz_i =0$ where $a_i$ is the cardinality 
of the set $ \{ y_1,\dots,y_n \}  \cap  \psi_d^{-1}(z_i)$ 
for $1 \leq i \leq m/d$. On the other hand, $\psi_d^{-1}(z)$ has exactly $d$ elements for each $z \in \Omega_{m/d}$. 
It follows that $H(n;m/d) \leq \max \{ a_1,\; \dots, a_{m/d} \} \leq d$. This proves that if $H(n;\ell,m) = 1$ then $H(n;m/d) \leq d$. 

Conversely, suppose that $H(n;m/d) \leq d$. Then there is a partition $( a_1,\dots,a_{m/d} )$ of $n$ into non-negative 
integers $a_i$ with $a_i \leq d$ for $ 1 \leq i \leq m/d$ and $\sum_{i=1}^{m/d} a_iz_i =0$. Let $y_i$ be any element 
of $\psi_d^{-1}(z_i)$ for $1 \leq i \leq m/d$. Then $ \psi_d^{-1}(z_i)= \; y_i \Omega_d \; = \{ \; y_i x \; | \; x \in \Omega_d \}$. Since $a_i \leq d$, 
one can replace $a_iz_i$ by $y_i^d ( x_1^d+\cdots+x_{a_i}^d )$ where $x_1,\dots,x_{a_i}$ are distinct elements of $\Omega_d$. 
Hence $H(n;\ell,m)=H(n;d,m)=1$ since $\sum_{i=1}^{m/d} a_i=n$. This completes the proof of Theorem \ref{thm:3}.  
\end{proof}

\end{document}